\documentclass[12pt,twoside]{amsart}
\usepackage{amsmath, amsthm, amscd, amsfonts, amssymb, graphicx}
\usepackage{enumerate}
\usepackage[colorlinks=true,
linkcolor=blue,
urlcolor=cyan,
citecolor=red]{hyperref}
\usepackage{mathrsfs}
\newtheorem{theorem}{Theorem}[section]

\newtheorem{remark}{Remark}[section]

\newtheorem{corollary}{Corollary}[section]

\newtheorem{proposition}{Proposition}[section]
\numberwithin{equation}{section}

\begin{document}
	
\title{On The Operator Hermite--Hadamard Inequality}
\author{Hamid Reza Moradi, Mohammad Sababheh and Shigeru Furuichi}

\begin{abstract}
The main target of this paper is to discuss operator Hermite--Hadamard inequality for convex functions, without appealing to  operator convexity.  Several forms of this inequality will be presented and some applications including norm and mean inequalities will be shown too.
\end{abstract}
\maketitle
\pagestyle{myheadings}
\markboth{\centerline {On The Operator Hermite--Hadamard Inequality}}
{\centerline {Hamid Reza Moradi, Mohammad Sababheh and Shigeru Furuichi}}
\bigskip
\bigskip
\section{Introduction and preliminaries}
Let $\mathcal{B}\left( \mathcal{H} \right)$ be the $C^*$--algebra of all bounded linear operators on a Hilbert space $\mathcal{H}$.  As usual, we reserve $m$, $M$ for scalars and ${{\mathbf{1}}_{\mathcal{H}}}$ for the identity operator on $\mathcal{H}$. A self adjoint operator $A$ is said to be positive (written $A\ge0$) if $\left\langle Ax,x \right\rangle \ge 0$  for all $x\in \mathcal{H}$, while it is said to be strictly positive (written $A>0$) if $A$ is positive and invertible. If $A$ and $B$ are self adjoint, we write $B\ge A$ in case $B-A\ge0$.

The Gelfand map $f\left( t \right)\mapsto f\left( A \right)$ is an isometrical $*$--isomorphism between the ${{C}^{*}}$--algebra $C\left( {\text{sp}}\left( A \right) \right)$ of continuous functions on the spectrum ${\text{sp}}\left( A \right)$ of a self  adjoint operator $A$ and the ${{C}^{*}}$--algebra generated by $A$ and the identity operator ${{\mathbf{1}}_{\mathcal{H}}}$. If $f,g\in C\left( {\text{sp}}\left( A \right) \right)$, then $f\left( t \right)\ge g\left( t \right)$ ($t\in {\text{sp}}\left( A \right)$) implies that $f\left( A \right)\ge g\left( A \right)$. This is called the functional calculus for the operator $A$.

 A real valued continuous function $f$ defined on the interval $J$ is said to be operator convex if $f\left( \left( 1-v \right)A+vB \right)\le \left( 1-v \right)f\left( A \right)+vf\left( B \right)$ for every $0<v<1$ and for every pair of bounded self  adjoint operators $A$ and $B$ whose spectra are both in $J$. One of the most important examples is the power function $t\mapsto {{t}^{p}}$ for $1\le p\le 2$. 

The Hermite--Hadamard inequality, named after Charles Hermite and Jacques Hadamard, states that if a function $f:J\to \mathbb{R}$ is convex, then the following chain of inequalities hold:
\begin{equation}\label{14}
f\left( \frac{a+b}{2} \right)\le \frac{1}{b-a}\int_{a}^{b}{f\left( t \right)dt}\le \frac{f\left( a \right)+f\left( b \right)}{2},\text{ }\left( a,b\in J, ~ a<b \right).
\end{equation}
Since (see, e.g. \cite[Lemma 2.1]{3})
\[\frac{1}{b-a}\int_{a}^{b}{f\left( x \right)dx}=\int_{0}^{1}{f\left( \left( 1-t \right)a+tb \right)dt}=\int_{0}^{1}{f\left( \left( 1-t \right)b+ta \right)dt},\]
we can rewrite \eqref{14} in the following form
\begin{equation}\label{her_t_form}
f\left( \frac{a+b}{2} \right)\le \int_{0}^{1}{f\left( \left( 1-t \right)a+tb \right)dt}\le \frac{f\left( a \right)+f\left( b \right)}{2}.
\end{equation}

The Hermite--Hadamard inequality plays an essential role in research on inequalities and has quite a sizeable technical literature; as one can see in \cite{6, 7, 8, 9, 1, 5}.

Obtaining operator inequalities corresponding to certain scalar inequalities have been an active research area in operator theory. Dragomir \cite{4} gave an operator version of Hermite--Hadamard inequality and proved that
\begin{equation}\label{18}
f\left( \frac{A+B}{2} \right)\le \int_{0}^{1}{f\left( \left( 1-t \right)A+tB \right)dt}\le \frac{f\left( A \right)+f\left( B \right)}{2},
\end{equation}
whenever $f:J\to \mathbb{R}$ is an operator convex and $A,B$ are two self adjoint operators with spectra  in $J$.\\
We emphasize here that the assumption {\it{operator convexity}} is essential to obtain \eqref{18}. For
example, if 
$$A=\left( \begin{matrix}
2 & 1  \\
1 & 1  \\
\end{matrix} \right), \quad B=\left( \begin{matrix}
1 & 0  \\
0 & 0  \\
\end{matrix} \right)\quad{\text{and}}\quad f\left( t \right)={{t}^{3}},$$ then simple computations show that 
$$f\left( \frac{A+B}{2} \right)=\left( \begin{matrix}
{17}/{4}\; & {7}/{4}\;  \\
{7}/{4}\; & {3}/{4}\;  \\
\end{matrix} \right),\quad \frac{f\left( A \right)+f\left( B \right)}{2}=\left( \begin{matrix}
7 & 4  \\
4 & {5}/{2}\;  \\
\end{matrix} \right)$$ and $$\int_{0}^{1}{f\left( \left( 1-t \right)A+tB \right)dt}=\left( \begin{matrix}
{31}/{6}\; & {5}/{2}\;  \\
{5}/{2}\; & {4}/{3}\;  \\
\end{matrix} \right).$$ It is easily seen that
\[f\left( \frac{A+B}{2} \right)\nleq \int_{0}^{1}{f\left( \left( 1-t \right)A+tB \right)dt}\nleq \frac{f\left( A \right)+f\left( B \right)}{2}.\]
So, even though $f(t)=t^3$ is convex (not operator convex), \eqref{18} does not hold; showing that operator convexity cannot be dropped.

It is then natural to ask about which conditions one should have so that the inequalities in \eqref{18} are valid for any convex function.

In \cite{sab_subadditive}, it is shown that convex functions satisfy \eqref{18} if some empty intersection conditions are imposed on the spectra of $A,B$. In this article, we present several forms of \eqref{18} using the  Mond--Pe\v cari\'c method for convex functions. For example, we show that for appropriate constants $\alpha, \beta,$
\begin{equation}\label{first_ineq_intro}
\int_{0}^{1}{f\left( \left( 1-t \right)A+tB \right)dt}\le \beta \mathbf{1}_\mathcal{H} +\alpha \left( \frac{g\left( A \right)+g\left( B \right)}{2} \right),
\end{equation}
 when $m{{\mathbf{1}}_{\mathcal{H}}}\le A,B\le M{{\mathbf{1}}_{\mathcal{H}}}$ and $f,g$ are certain functions. Then several converses and variants of \eqref{first_ineq_intro} are presented. See Theorem \ref{21} and the results that follow for the details.

 In the end, we present other forms using properties of inner product; without appealing to the  Mond--Pe\v cari\'c method.  Our results generalize some known inequalities presented in \cite{4, 1}.\\
 In our proofs, we will frequently use the basic inequality \cite[Theorem 1.2]{2}
 \begin{equation}\label{inner_conv_intro}
 f\left(\left<Ax,x\right>\right)\leq \left<f(A)x,x\right>
 \end{equation}
valid for  the convex function $f:J\to\mathbb{R}$, the self adjoint operator $A$ with spectrum in $J$ and the unit vector $x\in \mathcal{H}.$
 \section{Main Results}
We present our main results in to sections; where the  Mond--Pe\v cari\'c method is discussed first.
  Throughout this section, we use the following two standard notations for the function $f:[m,M]\to\mathbb{R},$
\[{{a}_{f}}=\frac{f\left( M \right)-f\left( m \right)}{M-m}\text{  }\!\!\And\!\!\text{  }{{b}_{f}}=\frac{Mf\left( m \right)-mf\left( M \right)}{M-m}.\]
 
 \subsection{Hermite--Hadamard inequalities using the Mond--Pe\v cari\'c method}

Our first convex (not operator convex) version of \eqref{18} reads as follows.
\begin{theorem}\label{21}
Let $A,B\in \mathcal{B}\left( \mathcal{H} \right)$ be two self adjoint operators satisfying $m{{\mathbf{1}}_{\mathcal{H}}}\le A,B\le M{{\mathbf{1}}_{\mathcal{H}}}$ and let $f,g:\left[ m,M \right]\to \mathbb{R}$ be two continuous functions. If $f$ and $g$ are both convex functions, then for a given $\alpha\ge 0,$
\begin{equation}\label{7}
\int_{0}^{1}{f\left( \left( 1-t \right)A+tB \right)dt}\le \beta \mathbf{1}_\mathcal{H} +\alpha \left( \frac{g\left( A \right)+g\left( B \right)}{2} \right),
\end{equation}	
where $\beta=\underset{m\le x\le M}{\mathop{\max }}\,\left\{ {{a}_{f}}x+{{b}_{f}}-\alpha g\left( x \right) \right\}.$
\end{theorem}
\begin{proof}
It follows from the convexity of $f:\left[ m,M \right]\to \mathbb{R}$ that 
\begin{equation}\label{1}
f\left( x \right)\le {{a}_{f}}x+{{b}_{f}}
\end{equation}
for any $m\le x\le M$. Since $m{{\mathbf{1}}_{\mathcal{H}}}\le A,B\le M{{\mathbf{1}}_{\mathcal{H}}}$, then $m{{\mathbf{1}}_{\mathcal{H}}}\le \left( 1-t \right)A+tB\le M{{\mathbf{1}}_{\mathcal{H}}}$.
Applying functional calculus for the operator $T=\left( 1-t \right)A+tB$ in \eqref{1} implies
\[f\left( \left( 1-t \right)A+tB \right)\le {{a}_{f}}\left( \left( 1-t \right)A+tB \right)+{{b}_{f}}\mathbf{1}_\mathcal{H}.\]
Integrating the inequality over $t\in \left[ 0,1 \right]$, we get
\[\int_{0}^{1}{f\left( \left( 1-t \right)A+tB \right)dt}\le {{a}_{f}}\left( \frac{A+B}{2} \right)+{{b}_{f}}\mathbf{1}_\mathcal{H}.\]
Now, let $x\in \mathcal{H}$ be a unit vector. One can write
\begin{align}
& \left\langle \left(\int_{0}^{1}{f\left( \left( 1-t \right)A+tB \right)dt}\right)x,x \right\rangle -\alpha \left\langle \left( \frac{g\left( A \right)+g\left( B \right)}{2} \right)x,x \right\rangle \nonumber \\ 
& \le {{a}_{f}}\left\langle \left( \frac{A+B}{2} \right)x,x \right\rangle +{{b}_{f}}-\alpha \left\langle \left( \frac{g\left( A \right)+g\left( B \right)}{2} \right)x,x \right\rangle  \nonumber\\ 
& ={{a}_{f}}\left\langle \left( \frac{A+B}{2} \right)x,x \right\rangle +{{b}_{f}}-\alpha \left( \frac{\left\langle g\left( A \right)x,x \right\rangle +\left\langle g\left( B \right)x,x \right\rangle }{2} \right) \nonumber\\ 
& \le {{a}_{f}}\left\langle \left( \frac{A+B}{2} \right)x,x \right\rangle +{{b}_{f}}-\alpha \left( \frac{g\left( \left\langle Ax,x \right\rangle  \right)+g\left( \left\langle Bx,x \right\rangle  \right)}{2} \right) \label{10}\\ 
& \le {{a}_{f}}\left\langle \left( \frac{A+B}{2} \right)x,x \right\rangle +{{b}_{f}}-\alpha g\left( \frac{\left\langle Ax,x \right\rangle +\left\langle Bx,x \right\rangle }{2} \right) \label{11}\\ 
& ={{a}_{f}}\left\langle \left( \frac{A+B}{2} \right)x,x \right\rangle +{{b}_{f}}-\alpha g\left( \left\langle \left(\frac{A+B}{2}\right)x,x \right\rangle  \right) \nonumber\\ 
& \le \underset{m\le x\le M}{\mathop{\max }}\,\left\{ {{a}_{f}}x+{{b}_{f}}-\alpha g\left( x \right) \right\}\nonumber
\end{align}
where in \eqref{10} we used \eqref{inner_conv_intro}, and \eqref{11} follows directly from  convexity of $g$.\\
Consequently,
\[\left\langle \left( \int_{0}^{1}{f\left( \left( 1-t \right)A+tB \right)dt} \right)x,x \right\rangle \le \beta +\alpha \left\langle \left( \frac{g\left( A \right)+g\left( B \right)}{2} \right)x,x \right\rangle \]
for any unit vector $x\in \mathcal{H}$. This completes the proof of inequality \eqref{7}. 
\end{proof}

Now we present some applications of Theorem \ref{21}. 
\begin{corollary}
Let $A,B\in \mathcal{B}\left( \mathcal{H} \right)$ be two self adjoint operators satisfying $m{{\mathbf{1}}_{\mathcal{H}}}\le A,B\le M{{\mathbf{1}}_{\mathcal{H}}}$ and let $f,g:\left[ m,M \right]\to \mathbb{R}$ be two continuous functions. If $f$ and $g>0$ are convex, then
\begin{equation}\label{22}
\int_{0}^{1}{f\left( \left( 1-t \right)A+tB \right)dt}\le \alpha \left( \frac{g\left( A \right)+g\left( B \right)}{2} \right),
\end{equation}
where $\alpha =\underset{m\le x\le M}{\mathop{\max }}\,\left\{ \frac{{{a}_{f}}x+{{b}_{f}}}{g\left( x \right)} \right\}$.

Further,
\[\int_{0}^{1}{f\left( \left( 1-t \right)A+tB \right)dt}\le \beta {{\mathbf{1}}_{\mathcal{H}}}+\frac{g\left( A \right)+g\left( B \right)}{2},\]
where $\beta =\underset{m\le x\le M}{\mathop{\max }}\,\left\{ {{a}_{f}}x+{{b}_{f}}-g\left( x \right) \right\}$
\end{corollary}
\begin{proof}
Notice that when $\alpha =\underset{m\le x\le M}{\mathop{\max }}\,\left\{ \frac{{{a}_{f}}x+{{b}_{f}}}{g\left( x \right)} \right\}$, then $a_fx+b_f-\alpha g(x)\leq 0.$   Therefore, from Theorem \ref{21}, $\beta\leq 0$ and \eqref{7} implies \eqref{22}. The other inequality follows similarly from Theorem \ref{21}.
\end{proof}
\begin{remark}
Setting $f=g>0$ the inequality \eqref{22} implies
\begin{equation}\label{20}
\int_{0}^{1}{f\left( \left( 1-t \right)A+tB \right)dt}\le \alpha \left( \frac{f\left( A \right)+f\left( B \right)}{2} \right)
\end{equation}
where $\alpha =\underset{m\le x\le M}{\mathop{\max }}\,\left\{ \frac{{{a}_{f}}x+{{b}_{f}}}{f\left( x \right)} \right\}$.
We remark that a similar result as in \eqref{20} was shown in \cite[Theorem 3.9]{1}. Therefore, Theorem \ref{21} can be considered as an extension of \cite[Theorem 3.9]{1}. 
\end{remark}
Notice that Theorem \ref{21} and its consequences above present operator order inequalities. In the next result, we obtain operator norm inequalities. Here, $|A|=(A^*A)^{1/2},$ where $A^*$ is the adjoint operator of $A$.
\begin{proposition}
Let $A,B\in \mathcal{B}\left( \mathcal{H} \right)$ be two self adjoint operators satisfying $m{{\mathbf{1}}_{\mathcal{H}}}\le \left| A \right|,\left| B \right|\le M{{\mathbf{1}}_{\mathcal{H}}}$ and let $f:\left[ m,M \right]\to \mathbb{R}$ be a nonnegative continuous increasing convex function. Then for a given $\alpha\ge 0,$
\[f\left( \left\| \frac{A+B}{2} \right\| \right)\le \left\| \int_{0}^{1}{f\left( \left( 1-t \right)\left| A \right|+t\left| B \right| \right)dt} \right\|\le \beta +\alpha \left\| \frac{f\left( \left| A \right| \right)+f\left( \left| B \right| \right)}{2} \right\|\]	
where $\beta =\underset{m\le x\le M}{\mathop{\max }}\,\left\{ {{a}_{f}}x+{{b}_{f}}-\alpha f\left( x \right) \right\}$.
\end{proposition}
\begin{proof}
Recall that if $T\in \mathcal{B}\left( \mathcal{H} \right)$  is a self adjoint operator, then $\left\| T \right\|=\underset{\left\| x \right\|=1}{\mathop{\sup }}\,\left| \left\langle Tx,x \right\rangle  \right|$. Let $x\in \mathcal{H}$ be a unit vector. Then
\begin{equation}\label{12}
\begin{aligned}
 f\left( \left| \left\langle \left( \frac{A+B}{2} \right)x,x \right\rangle  \right| \right)&=f\left( \left| \frac{\left\langle Ax,x \right\rangle +\left\langle Bx,x \right\rangle }{2} \right| \right) \\ 
& \le f\left( \frac{\left| \left\langle Ax,x \right\rangle  \right|+\left| \left\langle Bx,x \right\rangle  \right|}{2} \right) \quad \text{(by the triangle inequality)}\\ 
& \le f\left( \frac{\left\langle \left| A \right|x,x \right\rangle +\left\langle \left| B \right|x,x \right\rangle }{2} \right) \quad(\text{by }\eqref{inner_conv_intro})\\ 
& \le \int_{0}^{1}{f\left( \left( 1-t \right)\left\langle \left| A \right|x,x \right\rangle +t\left\langle \left| B \right|x,x \right\rangle  \right)dt}\quad({\text{by }}\eqref{her_t_form}) \\ 
& = \int_{0}^{1}{f\left( \left\langle \left( \left( 1-t \right)\left| A \right|+t\left| B \right| \right)x,x \right\rangle \right)dt}\\
&\leq  \int_0^1 \left\langle f\left( \left( 1-t \right)\left| A \right|+t\left| B \right| \right)x,x \right\rangle
dt  \quad({\text{by }}\eqref{inner_conv_intro}) \\
& =   \left\langle \left(\int_{0}^{1}{f\left( \left( 1-t \right)\left| A \right|+t\left| B \right| \right)dt}\right)x,x \right\rangle  \\ 
& \le  \left\| \int_{0}^{1}{f\left( \left( 1-t \right)\left| A \right|+t\left| B \right| \right)dt} \right\|. 
\end{aligned}
\end{equation}
  Now, by taking supremum over $x\in \mathcal{H}$ with $\left\| x \right\|=1$ in \eqref{12}  and noting that $f$ is increasing, 
\[\begin{aligned}
 f\left( \left\| \frac{A+B}{2} \right\| \right)&\le \left\| \int_{0}^{1}{f\left( \left( 1-t \right)\left| A \right|+t\left| B \right| \right)dt} \right\| \\ 
& \le \left\| \beta {{1}_{H}}+\alpha \left( \frac{f\left( \left| A \right| \right)+f\left( \left| B \right| \right)}{2} \right) \right\| \\ 
& \le \beta +\alpha \left\| \frac{f\left( \left| A \right| \right)+f\left( \left| B \right| \right)}{2} \right\|  
\end{aligned}\]
thanks to \eqref{7}. This completes the proof.
\end{proof}

We end this section by giving the weighted generalization of operator Hermite--Hadamard inequality. For convenience, we use  $A{{\nabla }_{\lambda }B}$ to denote $\left( 1-\lambda  \right)A+\lambda B$. We then show that Theorem \ref{nabla_thm} is a generalization of \eqref{18}.
\begin{theorem}\label{nabla_thm}
Let $A,B\in \mathcal{B}\left( \mathcal{H} \right)$ be two self adjoint operators satisfying $m{{\mathbf{1}}_{\mathcal{H}}}\le A,B\le M{{\mathbf{1}}_{\mathcal{H}}}$ and let $f:\left[ m,M \right]\to \mathbb{R}$ be an operator convex function. Then for any $0\le \lambda \le 1$,
\[\begin{aligned}
 f\left( A{{\nabla }_{\lambda }}B \right)&\le \int_{0}^{1}{f\left( \left( A{{\nabla }_{\lambda }}B \right){{\nabla }_{v}}A \right){{\nabla }_{\lambda }}f\left( \left( A{{\nabla }_{\lambda }}B \right){{\nabla }_{v}}B \right)dv} \\ 
& \le f\left( A \right){{\nabla }_{\lambda }}f\left( B \right).  
\end{aligned}\]
\end{theorem}
\begin{proof}
Since for $0\le \lambda ,v\le 1$,
\[A{{\nabla }_{\lambda }}B=\left( \left( A{{\nabla }_{\lambda }}B \right){{\nabla }_{v}}A \right){{\nabla }_{\lambda }}\left( \left( A{{\nabla }_{\lambda }}B \right){{\nabla }_{v}}B \right)\]
holds, we infer from the operator convexity of $f$ that
\[\begin{aligned}
 f\left( A{{\nabla }_{\lambda }}B \right)&=f\left( \left( \left( A{{\nabla }_{\lambda }}B \right){{\nabla }_{v}}A \right){{\nabla }_{\lambda }}\left( \left( A{{\nabla }_{\lambda }}B \right){{\nabla }_{v}}B \right) \right) \\ 
& \le f\left( \left( A{{\nabla }_{\lambda }}B \right){{\nabla }_{v}}A \right){{\nabla }_{\lambda }}f\left( \left( A{{\nabla }_{\lambda }}B \right){{\nabla }_{v}}B \right) \\ 
&\leq \left\{f(A\nabla_{\lambda}B)\nabla_v f(A)\right\}\nabla_{\lambda}\left\{f(A\nabla_{\lambda}B)\nabla_v f(B)\right\}\\
&\leq \left\{\left(f(A)\nabla_{\lambda}f(B)\right)\nabla_v f(A)\right\}\nabla_{\lambda}\left\{\left(f(A)\nabla_{\lambda}f(B)\right)\nabla_{v}f(B)\right\}                 \\
& \le f\left( A \right){{\nabla }_{\lambda }}f\left( B \right).
\end{aligned}\]
Integrating the inequality over $v\in \left[ 0,1 \right]$, we get
\[\begin{aligned}
 f\left( A{{\nabla }_{\lambda }}B \right)&\le \int_{0}^{1}{f\left( \left( A{{\nabla }_{\lambda }}B \right){{\nabla }_{v}}A \right){{\nabla }_{\lambda }}f\left( \left( A{{\nabla }_{\lambda }}B \right){{\nabla }_{v}}B \right)dv} \\ 
& \le f\left( A \right){{\nabla }_{\lambda }}f\left( B \right)  
\end{aligned}\]
which is the statement of the theorem. 
\end{proof}
\begin{remark}
To show that Theorem \ref{nabla_thm} is a generalization of \eqref{18}, put $\lambda ={1}/{2}\;$. Thus
\begin{equation}\label{19}
\begin{aligned}
& f\left( \frac{A+B}{2} \right) \\ 
& \le \frac{1}{2}\left[ \int_{0}^{1}{f\left( \left( 1-v \right)\left( \frac{A+B}{2} \right)+vA \right)dv}+\int_{0}^{1}{f\left( \left( 1-v \right)\left( \frac{A+B}{2} \right)+vB \right)dv} \right] \\ 
& \le \frac{f\left( A \right)+f\left( B \right)}{2}.
\end{aligned}
\end{equation}
On making use of the change of variable $v=1-2t$ we have
\begin{equation}\label{16}
\frac{1}{2}\int_{0}^{1}{f\left( \left( 1-v \right)\left( \frac{A+B}{2} \right)+vA \right)dv}=\int_{0}^{\frac{1}{2}}{f\left( \left( 1-t \right)A+tB \right)dt}.
\end{equation}
and by the change of variable $v=2t-1$,
\begin{equation}\label{17}
\frac{1}{2}\int_{0}^{1}{f\left( \left( 1-v \right)\left( \frac{A+B}{2} \right)+vB \right)dv}=\int_{\frac{1}{2}}^{1}{f\left( \left( 1-t \right)A+tB \right)dt}.
\end{equation}
Relations \eqref{16} and \eqref{17}, gives
\begin{equation}\label{15}
\begin{aligned}
& \frac{1}{2}\left[ \int_{0}^{1}{f\left( \left( 1-v \right)\left( \frac{A+B}{2} \right)+vA \right)dv}+\int_{0}^{1}{f\left( \left( 1-v \right)\left( \frac{A+B}{2} \right)+vB \right)dv} \right] \\ 
& =\int_{0}^{1}{f\left( \left( 1-t \right)A+tB \right)dt} \\ 
\end{aligned}
\end{equation}
and the assertion follows by combining \eqref{19} and \eqref{15}.
\end{remark}

\subsection{Reverse Hermite--Hadamard inequalities using the Mond--Pe\v cari\'c method}
In the forthcoming theorem, we give additive, and multiplicative type reverses for the first and the second inequalities in \eqref{18}.
\begin{theorem}
	Let $A,B\in \mathcal{B}\left( \mathcal{H} \right)$ be two self adjoint operators satisfying $m{{\mathbf{1}}_{\mathcal{H}}}\le A,B\le M{{\mathbf{1}}_{\mathcal{H}}}$ and let $f,g:\left[ m,M \right]\to \mathbb{R}$ be two continuous functions. If $f$ is a convex function, then for a given $\alpha\ge 0$
	\begin{equation}\label{2}
	\int_{0}^{1}{f\left( \left( 1-t \right)A+tB \right)dt}\le \beta {{\mathbf{1}}_{\mathcal{H}}}+\alpha g\left( \frac{A+B}{2} \right),
	\end{equation}
	and
	\begin{equation}\label{3}
	\frac{f\left( A \right)+f\left( B \right)}{2}\le \beta {{\mathbf{1}}_{\mathcal{H}}}+\alpha \int_{0}^{1}{g\left( \left( 1-t \right)A+tB \right)dt},
	\end{equation}
	where $\beta =\underset{m\le x\le M}{\mathop{\max }}\,\left\{ {{a}_{f}}x+{{b}_{f}}-\alpha g\left( x \right) \right\}$.
\end{theorem}
\begin{proof}
From \eqref{1} and by applying functional calculus for the operator $T=\left( 1-t \right)A+tB$, we have 
	\[f\left( \left( 1-t \right)A+tB \right)\le {{a}_{f}}\left( \left( 1-t \right)A+tB \right)+{{b}_{f}}{{\mathbf{1}}_{\mathcal{H}}}.\] 
	Integrating both sides of the above inequality over $t\in\left[ 0,1 \right]$, we have 
	\[\int_{0}^{1}{f\left( \left( 1-t \right)A+tB \right)dt}\le {{a}_{f}}\left( \frac{A+B}{2} \right)+{{b}_{f}}{{\mathbf{1}}_{\mathcal{H}}}.\] 
	Therefore,
	\[\begin{aligned}
	& \int_{0}^{1}{f\left( \left( 1-t \right)A+tB \right)dt}-\alpha g\left( \frac{A+B}{2} \right) \\ 
	& \le {{a}_{f}}\left( \frac{A+B}{2} \right)+{{b}_{f}}{{\mathbf{1}}_{\mathcal{H}}}-\alpha g\left( \frac{A+B}{2} \right) \\ 
	& \le \underset{m\le x\le M}{\mathop{\max }}\,\left\{ {{a}_{f}}x+{{b}_{f}}-\alpha g\left( x \right) \right\}\mathbf{1}_\mathcal{H}.
	\end{aligned}\] 
	Consequently,
	\[\int_{0}^{1}{f\left( \left( 1-t \right)A+tB \right)dt}\le \beta {{\mathbf{1}}_{\mathcal{H}}}+\alpha g\left( \frac{A+B}{2} \right)\] 
which proves \eqref{2}. To prove \eqref{3}, notice that \eqref{1} implies, for $0\leq t\leq 1,$
\begin{equation}\label{5}
(1-t)f\left( A \right)\le {{a}_{f}}(1-t)A+{{b}_{f}}(1-t){{\mathbf{1}}_{\mathcal{H}}},
\end{equation}
\begin{equation}\label{6}
tf\left( B \right)\le {{a}_{f}}tB+{{b}_{f}}t{{\mathbf{1}}_{\mathcal{H}}}.
\end{equation}
From \eqref{5} and \eqref{6} we infer that
\[(1-t)f\left( A \right)+tf\left( B \right)\le {{a}_{f}}\left( (1-t)A+tB \right)+{{b}_{f}}{{\mathbf{1}}_{\mathcal{H}}}.\]
Therefore
\[\begin{aligned}
& (1-t)f\left( A \right)+tf\left( B \right)-\alpha g\left( \left( 1-t \right)A+tB \right) \\ 
& \le {{a}_{f}}\left( (1-t)A+tB \right)+{{b}_{f}}{{\mathbf{1}}_{\mathcal{H}}}-\alpha g\left( \left( 1-t \right)A+tB \right) \\ 
& \le \underset{m\le x\le M}{\mathop{\max }}\,\left\{ {{a}_{f}}x+{{b}_{f}}-\alpha g\left( x \right) \right\}\mathbf{1}_\mathcal{H}.
\end{aligned}\]
Thus,
\begin{equation}\label{4}
(1-t)f\left( A \right)+tf\left( B \right) \le \beta {{\mathbf{1}}_{\mathcal{H}}}+\alpha g\left( \left( 1-t \right)A+tB \right).
\end{equation}
Integrating both sides of \eqref{4} over $\left[ 0,1 \right]$ we get  \eqref{3} and the proof is complete.
\end{proof}

\subsection{Operator Hermite--Hadamard inequality using the gradient inequality}
In this subsection, we present versions of the operator Hermite--Hadamard inequality using the gradient inequality
\begin{equation}\label{grad_ineq}
f'\left( s \right)\left( t-s \right)+f\left( s \right)\le f\left( t \right),
\end{equation} 
where $f:J\to\mathbb{R}$ is convex differentiable and $s,t\in J.$
\begin{theorem}\label{8}
	Let $A,B\in\mathcal{B}(\mathcal{H})$ be self adjoint operators with spectra in the interval $J$ and let $f:J\to\mathbb{R}$ be a differentiable convex function. Then
	
	\begin{equation}\label{24}
	f\left( \frac{A+B}{2} \right)\le \int_{0}^{1}{f\left( \left( 1-v \right)A+vB \right)dv}+\delta \mathbf{1}_\mathcal{H},
	\end{equation}
	where
	\[\delta =\underset{\left\| x \right\|=1}{\mathop{\underset{x\in \mathcal{H}}{\mathop{\sup }}\,}}\,\left\{ \left\langle f'\left( \frac{A+B}{2} \right)\left(\frac{A+B}{2}\right)x,x \right\rangle -\left\langle f'\left( \frac{A+B}{2} \right)x,x \right\rangle \left\langle \left( \frac{A+B}{2} \right)x,x \right\rangle  \right\}.\] 
\end{theorem}
\begin{proof}
	 Since $f$ is convex differentiable, \eqref{grad_ineq} applies. 
	By applying functional calculus for the operator $s=\frac{A+B}{2}$ we get
	\[tf'\left( \frac{A+B}{2} \right)-f'\left( \frac{A+B}{2} \right)\left(\frac{A+B}{2}\right)+f\left( \frac{A+B}{2} \right)\le f\left( t \right){{\mathbf{1}}_{\mathcal{H}}}.\] 
	So, for any unit vector $x\in \mathcal{H}$,
	\[t\left\langle f'\left( \frac{A+B}{2} \right)x,x \right\rangle -\left\langle f'\left( \frac{A+B}{2} \right)\left(\frac{A+B}{2}\right)x,x \right\rangle +\left\langle f\left( \frac{A+B}{2} \right)x,x \right\rangle \le f\left( t \right).\] 
	Again, by applying functional calculus for the operator $t=\left( 1-v \right)A+vB$ we get
	\[\begin{aligned}
	& \left\langle f'\left( \frac{A+B}{2} \right)x,x \right\rangle \left( (1-v )A+vB \right)-\left\langle f'\left( \frac{A+B}{2} \right)\left(\frac{A+B}{2}\right)x,x \right\rangle {{\mathbf{1}}_{\mathcal{H}}}+\left\langle f\left( \frac{A+B}{2} \right)x,x \right\rangle {{\mathbf{1}}_{\mathcal{H}}} \\ 
	& \le f\left( \left( (1-v)A+vB \right) \right). \\ 
	\end{aligned}\] 
	Integrating both sides over $t\in \left[ 0,1 \right]$ implies
	\[\begin{aligned}
	& \left\langle f'\left( \frac{A+B}{2} \right)x,x \right\rangle \left( \frac{A+B}{2} \right)-\left\langle f'\left( \frac{A+B}{2} \right)\left(\frac{A+B}{2}\right)x,x \right\rangle {{\mathbf{1}}_{\mathcal{H}}}+\left\langle f\left( \frac{A+B}{2} \right)x,x \right\rangle {{\mathbf{1}}_{\mathcal{H}}} \\ 
	& \le \int_{0}^{1}{ f\left( \left( 1-v \right)A+vB \right)dv}. \\ 
	\end{aligned}\]
	Whence, for any unit vector $x\in \mathcal{H}$,
	\[\begin{aligned}
	& \left\langle f'\left( \frac{A+B}{2} \right)x,x \right\rangle \left\langle \left( \frac{A+B}{2} \right)x,x \right\rangle -\left\langle f'\left( \frac{A+B}{2} \right)\left(\frac{A+B}{2}\right)x,x \right\rangle +\left\langle f\left( \frac{A+B}{2} \right)x,x \right\rangle  \\ 
	& \le \left\langle{ \left( \int_{0}^{1}f\left( ( 1-v )A+vB \right)dv \right)}x,x \right\rangle . \\ 
	\end{aligned}\]
	Thus,
	\[\left\langle f\left( \frac{A+B}{2} \right)x,x \right\rangle \le \left\langle \left(\int_{0}^{1} f\left( ( 1-v )A+vB \right)dv\right)x,x \right\rangle +\delta \] 
	where
	\[\delta =\underset{\left\| x \right\|=1}{\mathop{\underset{x\in \mathcal{H}}{\mathop{\sup }}\,}}\,\left\{ \left\langle f'\left( \frac{A+B}{2} \right)\left(\frac{A+B}{2}\right)x,x \right\rangle -\left\langle f'\left( \frac{A+B}{2} \right)x,x \right\rangle \left\langle \left( \frac{A+B}{2} \right)x,x \right\rangle  \right\}.\] 
	Therefore,
	
	\[f\left( \frac{A+B}{2} \right)\le \int_{0}^{1}{f\left( \left( 1-v \right)A+vB \right)dv}+\delta \mathbf{1}_\mathcal{H} ,\]
	which completes the proof. 
\end{proof}
Our last result in this direction is as follows.

\begin{theorem}\label{9}
	Let $A,B\in\mathcal{B}(\mathcal{H})$ be self adjoint operators with spectra in the interval $J$ and let $f:J\to\mathbb{R}$ be a differentiable convex function. Then
	\begin{equation}\label{25}
	\int_{0}^{1}{f\left( \left( 1-v \right)A+vB \right)dv}\le \frac{f\left( A \right)+f\left( B \right)}{2}+\xi \mathbf{1}_\mathcal{H},
	\end{equation}
	where
\[\begin{aligned}
& \xi =\underset{\left\| x \right\|=1}{\mathop{\underset{x\in \mathcal{H}}{\mathop{\sup }}\,}}\,\left\{ \int_{0}^{1}{\left\langle f'\left( \left( 1-v \right)A+vB \right)\left( \left( 1-v \right)A+vB \right)x,x \right\rangle dv} \right. \\ 
&\quad \quad\quad\quad \left. -\int_{0}^{1}{\left\langle f'\left( \left( 1-v \right)A+vB \right)x,x \right\rangle \left\langle \left( \left( 1-v \right)A+vB \right)x,x \right\rangle dv} \right\}.  
\end{aligned}\]
\end{theorem}
\begin{proof}
	By applying functional calculus for the operator $T=\left( 1-v \right)A+vB$ in \eqref{grad_ineq}, we have
	\[tf'\left( \left( 1-v \right)A+vB \right)-f'\left( \left( 1-v \right)A+vB \right)\left( \left( 1-v \right)A+vB \right)+f\left( \left( 1-v \right)A+vB \right)\le f\left( t \right)\mathbf{1}_\mathcal{H}.\]
	Hence for any unit vector $x\in \mathcal{H}$,
	\[\begin{aligned}
	& t\left\langle f'\left( \left( 1-v \right)A+vB \right)x,x \right\rangle -\left\langle f'\left( \left( 1-v \right)A+vB \right)\left( \left( 1-v \right)A+vB \right)x,x \right\rangle  \\ 
	&\quad\quad\quad\quad\quad +\left\langle f\left( \left( 1-v \right)A+vB \right)x,x \right\rangle \le f\left( t \right).
	\end{aligned}\]
	Again, it follows from the functional calculus for $t=A$ and $t=B$, respectively
	\begin{equation}\label{01}
	\begin{aligned}
	& \left( 1-v \right)\left\langle f'\left( \left( 1-v \right)A+vB \right)x,x \right\rangle A-\left( 1-v \right)\left\langle f'\left( \left( 1-v \right)A+vB \right)\left( \left( 1-v \right)A+vB \right)x,x \right\rangle\mathbf{1}_\mathcal{H}  \\ 
	&\quad +\left( 1-v \right)\left\langle f\left( \left( 1-v \right)A+vB \right)x,x \right\rangle\mathbf{1}_\mathcal{H} \le \left( 1-v \right)f\left( A \right),
	\end{aligned}
	\end{equation}
	and
	\begin{equation}\label{02}
	\begin{aligned}
	& v\left\langle f'\left( \left( 1-v \right)A+vB \right)x,x \right\rangle B-v\left\langle f'\left( \left( 1-v \right)A+vB \right)\left( \left( 1-v \right)A+vB \right)x,x \right\rangle\mathbf{1}_\mathcal{H}  \\ 
	&\quad +v\left\langle f\left( \left( 1-v \right)A+vB \right)x,x \right\rangle \mathbf{1}_\mathcal{H}\le vf\left( B \right).
	\end{aligned}
	\end{equation}
	By combining \eqref{01} and \eqref{02} we obtain
	\[\begin{aligned}
	& \left\langle f'\left( \left( 1-v \right)A+vB \right)x,x \right\rangle \left( \left( 1-v \right)A+vB \right)-\left\langle f'\left( \left( 1-v \right)A+vB \right)\left( \left( 1-v \right)A+vB \right)x,x \right\rangle\mathbf{1}_\mathcal{H}  \\ 
	&\quad +\left\langle f\left( \left( 1-v \right)A+vB \right)x,x \right\rangle \mathbf{1}_\mathcal{H}\le \left( 1-v \right)f\left( A \right)+vf\left( B \right).
	\end{aligned}\]
	This implies
	\[\begin{aligned}
	& \left\langle f'\left( \left( 1-v \right)A+vB \right)x,x \right\rangle \left\langle \left( \left( 1-v \right)A+vB \right)x,x \right\rangle  \\ 
	&\quad -\left\langle f'\left( \left( 1-v \right)A+vB \right)\left( \left( 1-v \right)A+vB \right)x,x \right\rangle +\left\langle f\left( \left( 1-v \right)A+vB \right)x,x \right\rangle  \\ 
	& \le \left\langle \left( \left( 1-v \right)f\left( A \right)+vf\left( B \right) \right)x,x \right\rangle  \\ 
	\end{aligned}\]
	for any unit vector $x\in \mathcal{H}$. Integrating both sides over $v\in \left[ 0,1 \right]$ we get
	\[\left\langle \left(\int_{0}^{1}f\left( \left( 1-v \right)A+vB \right)dv\right)x,x \right\rangle \le \left\langle \left( \frac{f\left( A \right)+f\left( B \right)}{2} \right)x,x \right\rangle +\xi \]
	where
\[\begin{aligned}
& \xi =\underset{\left\| x \right\|=1}{\mathop{\underset{x\in \mathcal{H}}{\mathop{\sup }}\,}}\,\left\{ \int_{0}^{1}{\left\langle f'\left( \left( 1-v \right)A+vB \right)\left( \left( 1-v \right)A+vB \right)x,x \right\rangle dv} \right. \\ 
&\quad \quad\quad\quad \left. -\int_{0}^{1}{\left\langle f'\left( \left( 1-v \right)A+vB \right)x,x \right\rangle \left\langle \left( \left( 1-v \right)A+vB \right)x,x \right\rangle dv} \right\}. \\ 
\end{aligned}\]
	Consequently,
	\[\int_{0}^{1}{f\left( \left( 1-v \right)A+vB \right)dv}\le \frac{f\left( A \right)+f\left( B \right)}{2}+\xi \mathbf{1}_\mathcal{H},\]
	as desired.
\end{proof}
\begin{remark}
Notice that in both Theorems \ref{8} and \ref{9}, a quantity of the form 
$$\sup_{\|x\|=1}\left\{\left<Af'(A)x,x\right>-\left<Ax,x\right>\left<f'(A)x,x\right>\right\}$$ has been found as a refining term, for some self adjoint operator $A$. We show here that this quantity is always non-negative, when $f$ is such a convex function.\\
Applying functional calculus for $s=A$ in \eqref{grad_ineq}, we obtain
$$f(A)-f(t)\mathbf{1}_\mathcal{H}\leq Af'(A)-tf'(A),$$ which implies
$$\left<f(A)x,x\right>-f(t)\leq \left<Af'(A)x,x\right>-t\left<f'(A)x,x\right>, x\in\mathcal{H}, \|x\|=1.$$ Now replacing $t$ by $\left<Ax,x\right>$ and noting \eqref{inner_conv_intro}, we obtain
$$ \left<Af'(A)x,x\right>-\left<Ax,x\right>\left<f'(A)x,x\right>\geq \left<f(A)x,x\right>-f\left(\left<Ax,x\right>\right)\geq 0,$$ as desired.
\end{remark}

\vskip 0.5 true cm

\vskip 0.3 true cm

{\tiny (H. R. Moradi) Department of Mathematics, Payame Noor University (PNU), P.O.Box, 19395-4697, Tehran, Iran.
	
	\textit{E-mail address:} hrmoradi@mshdiau.ac.ir}

\vskip 0.3 true cm

{\tiny (M. Sababheh) Department of Basic Sciences, Princess Sumaya University for Technology, Amman 11941,
	Jordan. 
	
	\textit{E-mail address:} sababheh@yahoo.com; sababheh@psut.edu.jo}

{\tiny (S. Furuichi) Department of Information Science, College of Humanities and Sciences, Nihon University, 3-25-40, Sakurajyousui,	Setagaya-ku, Tokyo, 156-8550, Japan. }

{\tiny \textit{E-mail address:} furuichi@chs.nihon-u.ac.jp }

\begin{thebibliography}{9}
\bibitem{6}
S. S. Dragomir and K. Nikodem, {\it Jensen's and Hermite--Hadamard's type inequalities for lower and strongly convex functions on normed spaces}, Bull. Iranian Math. Soc., {\bf44}(5) (2018), 1337--1349.

\bibitem{7}
S. S. Dragomir, {\it Hermite--Hadamard's type inequalities for convex functions of self adjoint operators in Hilbert spaces}, Linear Algebra Appl., {\bf436}(5) (2012), 1503--1515.	

\bibitem{4}
S. S. Dragomir, {\it Hermite--Hadamard's type inequalities for operator convex functions}, Appl. Math. Comput., {\bf218}(3) (2011), 766--772.

\bibitem{3}
A. El Farissi, {\it Simple proof and refinement of Hermite--Hadamard inequality}, J. Math. Ineq., {\bf4}(3) (2010), 365--369.
	

\bibitem{8}
S. Furuichi and H. R. Moradi, {\it Some refinements of classical inequalities}, Rocky Mountain J. Math., {\bf48}(7) (2018), 2289--2309.

\bibitem{2}
T. Furuta, J. Mi\'ci\'c, J. Pe\v cari\'c and Y. Seo, {\it Mond--Pe\v cari\'c method in operator inequalities}, Element, Zagreb, 2005.

\bibitem{sab_subadditive} H. R. Moradi,  Z. Heydarbeygi and M. Sababheh, {\it Subadditive inequalities for operators}, Math. Inequal. Appl., accepted.
\bibitem{9}
H. R. Moradi, S. Furuichi and N. Minculete, {\it Estimates for Tsallis relative operator entropy}, Math. Inequal. Appl., {\bf20}(4) (2017), 1079--1088.

\bibitem{1}
M. S. Moslehian, {\it Matrix Hermite--Hadamard type inequalities}, Houston J. Math., {\bf39}(1) (2013), 177--189.

\bibitem{5}
M. Sababheh and H. R. Moradi, {\it More accurate numerical radius inequalities}, Linear Multilinear Algebra., accepted.
\end{thebibliography}
\end{document}